\documentclass[10pt,leqno]{amsart}

\usepackage{graphicx}
\usepackage{indentfirst,csquotes}

\usepackage{amsfonts}
\usepackage{amsmath} 
\usepackage{amssymb}
\usepackage{array} 
\usepackage[lmargin = 1in, rmargin = 1in, tmargin = 1in, bmargin = 1in]{geometry}

\usepackage{xcolor}

\usepackage{xcolor}

\usepackage{blindtext}

\usepackage{amssymb,latexsym}
\theoremstyle{plain}
\newtheorem{theorem}{Theorem}
\newtheorem{corollary}{Corollary}

\newtheorem{lemma}{Lemma}

\theoremstyle{definition}

\theoremstyle{remark}

\newtheorem*{remark}{Remark}
\numberwithin{equation}{section}

\newcommand{\psum}{\sideset{}{^*}\sum}

\newcommand{\Ad}{\operatorname{Ad}}

\newcommand{\RP}{\text{RP}}
\newcommand{\cP}{\mathcal{P}}
\newcommand{\Sym}{\text{Sym}}
\newcommand{\GL}{\text{GL}}
\newcommand{\bA}{\mathbb{A}}
\newcommand{\bQ}{\mathbb{Q}}

 \newcommand\blfootnote[1]{%
	\begingroup
	\renewcommand\thefootnote{}\footnote{#1}%
	\addtocounter{footnote}{-1}%
	\endgroup
}

\title{On Ramanujan Primes for Hecke-Maass cusp forms}
\author{Tinghao Huang and Shifan Zhao}

\address{Math Building, 231 W. 18th Ave, Columbus, OH 43210 USA}
\email{huang.4939@osu.edu}
\address{Math Building, 231 W. 18th Ave, Columbus, OH 43210 USA}
\email{zhao.3326@osu.edu}

\begin{document}

\begin{abstract}
    For a primitive Hecke-Maass cusp form $\phi$ of level $N$ with the $n$-th Hecke eigenvalue $\lambda_{\phi}(n)$ and a prime number $p\nmid N$, the celebrated Ramanujan conjecture at $p$ asserts the following sharp upper bound:
\[
    |\lambda_{\phi}(p)| \leq 2.
\]
 In this work, 
 we determine an upper bound for the least prime $p$ at which the Ramanujan conjecture holds for two or three distinct primitive Hecke-Maass cusp forms simultaneously. Moreover, given a set of distinct primitive Hecke-Maass cusp forms $\{\phi_i\}$, we also provide a lower bound for the lower natural density of the set of primes at which the Ramanujan conjecture holds for at least one of the $\phi_i$'s.
\end{abstract}

\maketitle

\blfootnote{2020 Mathematics Subject Classification: 11F30, 11F66}
\blfootnote{Key Words: Ramanujan conjecture, Hecke-Maass cusp form, zero-free region}

\tableofcontents

\section{Introduction}
The celebrated Ramanujan conjecture for a primitive holomorphic cuspidal Hecke eigenform $f$ of weight $k \geq 2$ and level $N$ asserts that for primes $p \nmid N$, the following sharp bound for the $p$-th Hecke eigenvalue $\lambda_f(p)$ for $f$ holds:
\[
    |\lambda_f(p)| \leq 2.
\]
This conjecture has been solved affirmatively by Deligne in \cite{DeligneIMR340258} and \cite{DeligneIIMR3077124} as a consequence of his solution to the Weil conjectures. 

Now for a primitive Hecke-Maass cusp form $\phi$ of level $N$ with the $n$-th Hecke eigenvalue $\lambda_{\phi}(n)$ and a prime number $p\nmid N$, the Ramanujan conjecture for $\phi$ at $p$ predicts the following sharp upper bound:
\[
    |\lambda_{\phi}(p)| \leq 2.
\]
This conjecture, though unsolved, is believed to be true for any prime not dividing $N$, just as in the holomorphic case. This is an outstanding unsolved problem in number theory, which would follow from the solution to the Langlands functoriality conjectures.
The current best result towards the Ramanujan conjecture for $\phi$ is due to Kim and Sarnak \cite{Kim2003}, which states that for $p\nmid N$
\begin{equation}\label{KimSarnakboundforphi}
     |\lambda_{\phi}(p)| \leq p^{\frac{7}{64}} + p^{-\frac{7}{64}}.
\end{equation}

Let $\phi$ be a primitive Hecke-Maass cusp form of level $N$ and central character $\chi_{\phi}$. Define 
\begin{equation*}
    \RP(\phi) := \{p\nmid N: |\lambda_\phi(p)| \leq 2\} 
\end{equation*}
to be the set of primes $p$ at which the Ramanujan conjecture holds for $\phi$.
In \cite{LuoZhouMR3928043}, Luo and Zhou investigated two problems related to the Ramanujan conjecture for primitive Hecke-Maass cusp forms:
\begin{enumerate}
  \item Given $\phi$ of level $N$, can we find an upper bound for the \textit{least} $p \in \RP(\phi)$?
    \item Given $\phi$ of level $N$, can we find a lower bound for the \textit{lower natural density} for $\RP(\phi)$ in the set of all prime numbers?
\end{enumerate}    
The authors provided affirmative answers to both of the above questions, and we refer the readers to \cite[Theorem 1, Theorem 2]{LuoZhouMR3928043} for more details.
We also remark that both of the above problems have been investigated in higher rank settings, see for example \cite[Theorem B]{YangMR4549878} and \cite[Theorem 1.2]{LauNgWangMR4282084}.

The goal of this paper is to investigate the above questions in the context of multiple primitive Hecke-Maass cusp forms.

In the first part of this paper, we determine upper bounds for the least prime at which the Ramanujan conjecture is true for two or three distinct primitive Hecke-Maass cusp forms \textit{simultaneously}. 
Our results state the following:
\begin{theorem}\label{1Theorem1smallestRCprime}
    Let $\phi_1$ and $\phi_2$ be two primitive Hecke-Maass cusp forms of levels $N_1$ and $N_2$ and Laplacian eigenvalues $\frac{1}{4} + t_{\phi_1}^2$ and $\frac{1}{4} + t_{\phi_1}^2$, we may find a prime number $p\in \RP(\phi_1)\cap \RP(\phi_2)$ such that 
    \[
        p \ll [N_1N_2(1 + |t_{\phi_1}|)(1 + |t_{\phi_2}|)]^{0.447374}
    \]
    where the implied constant is absolute.
\end{theorem}

\begin{theorem}\label{2Theorem1smallestRCprime}
    Let $\phi_1, \phi_2$ and $\phi_3$ be three primitive Hecke-Maass cusp forms of levels $N_1,N_2$ and $N_3$ and Laplacian eigenvalues $\frac{1}{4} + t_{\phi_1}^2, \frac{1}{4} + t_{\phi_2}^2$ and $\frac{1}{4} + t_{\phi_3}^2$, we may find a prime number $p\in \RP(\phi_1)\cap\RP(\phi_2)\cap\RP(\phi_3)$ such that 
    \[
        p \ll [N_1N_2N_3(1 + |t_{\phi_1}|)(1 + |t_{\phi_2}|)(1 + |t_{\phi_3}|)]^{0.778798},
    \]
    where the implied constant is absolute.
\end{theorem}

In the second part of this paper, we shall investigate the (lower) natural density of the set of prime numbers at which the Ramanujan conjecture holds for at least one element in a given family of primitive Hecke-Maass cusp forms. 

Before stating our results we review some necessary backgrounds on natural densities of sets.
Let $\mathcal{P}$ be the set of prime numbers, and $\pi(X) := \sum_{p \leq X} 1$ be the prime counting function. For a subset $A \subset \cP$, recall that the upper natural density $\overline{d}(A)$ and the lower natural density $\underline{d}(A)$ of $A$ inside $\cP$ are defined by 
\begin{align*}
    \overline{d}(A) := \limsup_{X \to \infty} \frac{|\{p \leq X:p \in A\}|}{\pi(X)}, \\
    \underline{d}(A) := \liminf_{X \to \infty} \frac{|\{p \leq X:p \in A\}|}{\pi(X)},
\end{align*}
respectively. If $\overline{d}(A) = \underline{d}(A)$, then we define $d(A) := \overline{d}(A) = \underline{d}(A)$ and call it the natural density of $A$ inside $\cP$.

  In \cite[Theorem 1.2]{LuoZhouMR3928043}, Luo and Zhou proved
\begin{equation}\label{density result single form}
    \underline{d}(\RP(\phi)) \geq \frac{34}{35}.
\end{equation}
Now let $\phi_1$ and $\phi_2$ be two distinct primitive Hecke-Maass cusp forms of level $N_1$ and $N_2$ as before. Two natural questions on $\RP(\phi_1)$ and $\RP(\phi_2)$ arise:
\begin{enumerate}
    \item How large can $\underline{d}(\RP(\phi_1) \cap \RP(\phi_2))$ be? In other words, what is the lower natural density of the set of primes at which the Ramanujan conjecture holds for \textit{both $\phi_1$ and $\phi_2$}?
    \item How large can $\underline{d}(\RP(\phi_1) \cup \RP(\phi_2))$ be? In other words, what is the lower natural density of the set of primes at which the Ramanujan conjecture holds for \textit{at least one of $\phi_1$ and $\phi_2$}?
\end{enumerate}

Note that for either question above, it is natural to assume $\{p:p \nmid N_2,p \in \RP(\phi_1)\} \neq \{p:p \nmid N_1, p \in \RP(\phi_2)\}$. Since otherwise $\RP(\phi_1)$ and $\RP(\phi_2)$ differs only by a finite set of primes at most, and the question is no longer interesting.

For question (1) above, by the Pigeonhole principle and \eqref{density result single form} we have 
\begin{align*}
    \underline{d}(\RP(\phi_1)\cap \RP(\phi_2)) &\geq \underline{d}(\RP(\phi_1)) + \underline{d}(\RP(\phi_2)) - \underline{d}(\RP(\phi_1) \cup \RP(\phi_2)) \\
    &\geq \frac{34}{35}+\frac{34}{35}-1 = \frac{33}{35}.
\end{align*}

For question (2) above, by \eqref{density result single form} we have trivially 
\begin{equation}\label{density result two forms trivial}
    \underline{d}(\RP(\phi_1) \cup \RP(\phi_2)) \geq \max\{\underline{d}(\RP(\phi_1)),\underline{d}(\RP(\phi_2))\} \geq \frac{34}{35}.
\end{equation}

In section 3 we shall improve the bound \eqref{density result two forms trivial} by establishing the following theorem:

\begin{theorem}\label{main theorem density two forms}
    Let $\phi_1$ and $\phi_2$ be two primitive Hecke-Maass cusp forms of level $N_1$ and $N_2$. Assuming $\{p:p \nmid N_2,p \in \RP(\phi_1)\} \neq \{p:p \nmid N_1, p \in \RP(\phi_2)\}$, we have
    \begin{equation}
        \underline{d}(\RP(\phi_1) \cup \RP(\phi_2)) \geq \frac{43}{44}.
    \end{equation}
\end{theorem}

Our method of proving Theorem \ref{main theorem density two forms} generalizes from two forms to any finite number of primitive Hecke-Maass cusp forms:

\begin{theorem}\label{main theorem density finitely many forms}
    Let $\phi_i$ be primitive Hecke-Maass cusp forms for $i = 1, ...,m$ of levels $N_1, ...,N_m$ respectively. Assuming $\{p:p \nmid N_j,p \in \RP(\phi_i)\} \neq \{p:p \nmid N_i, p \in \RP(\phi_j)\}$ for any $i \neq j$, we have 
    \begin{equation}
        \underline{d}(\cup_{j=1}^m \RP(\phi_j)) \geq 1-\frac{1}{26+9m}.
    \end{equation}
\end{theorem}

Note that $\lim_{m \to \infty} \frac{1}{26+9m} = 0$. Thus a direct corollary of Theorem \ref{main theorem density finitely many forms} is the following:

\begin{corollary}\label{main result density infinitely many forms}
    Let $\{\phi_i\}_{i \in I}$ be an infinite family of primitive Hecke-Maass cusp forms, whose levels are respectively $N_i$'s for $i\in I$. Assuming $\{p:p \nmid N_j,p \in \RP(\phi_i)\} \neq \{p:p \nmid N_i, p \in \RP(\phi_j)\}$ for any $i,j \in I, i \neq j$, we have 
    \begin{equation}
        d(\cup_{i \in I} \RP(\phi_i))  = 1.
    \end{equation}
\end{corollary}

From Corollary \ref{main result density infinitely many forms} it can be stated that for almost all primes $p$ (in the sense of natural density), there exists some $i \in I$, such that $\phi_i$ is unramified at $p$, and the Ramanujan conjecture holds for $\phi_i$ at $p$.

\section{Proof of Theorem 1 and 2}
\subsection{Preliminaries}
Before proceeding the proof of the Theorem \ref{1Theorem1smallestRCprime} and Theorem \ref{2Theorem1smallestRCprime}, we review some necessary knowledge of the adjoint square $L$-functions of primitive Hecke-Maass cusp forms. 
Suppose $\phi$ is a primitive Hecke-Maass cusp form of level $N$ and central character $\chi_{\phi}$ whose $n$-th Hecke eigenvalue is denoted as $\lambda_{\phi}(n)$, then the adjoint square $L$-function of $\phi$, denoted by $L(s, \Ad(\phi))$, is defined as the following absolutely convergent Dirichlet series for $\Re(s) > 1$:
\[
    L(s, \Ad(\phi)) := \sum_{\substack{n \geq 1}}\frac{A_{\phi}(n)}{n^s}
\]
where $A_{\phi}(n) = \sum_{k^2|n}\lambda_{\phi}(n^2/k^4)\overline{\chi_{\phi}(n/k^2)}$ if $(n,N)=1$ and $A_{\phi}(n)= 0$ otherwise. The $L$-function $L(s, \Ad(\phi))$ could be analytically continued to the entire complex plane, which satisfies a functional equation of the $s \mapsto 1-s$ type.
Just as $\lambda_{\phi}(\cdot)$, the Dirichlet coefficient $A_{\phi}(\cdot)$ is also multiplicative.
The $L$-function satisfies an Euler product expansion for $\Re s > 1$:
\begin{equation}\label{EulerProductofAd}
    L(s,\Ad(\phi)) = \prod_{p\nmid N}\bigg[1 - \frac{A_{\phi}(p)}{p^s} + \frac{A_{\phi}(p)}{p^{2s}} - \frac{1}{p^{3s}}\bigg]^{-1}
\end{equation}

Suppose the Laplacian eigenvalue of $\phi$ is $\frac{1}{4} + t_{\phi}^2$, we shall set the conductor of $\phi$ as 
\[
    Q_{\phi} := N^2(1 + |t_{\phi}|)^2.
\]
Then, as $\Ad(\phi)$ is shown to be a cusp form for $GL(3)$ in \cite{GelbartJacquetMR533066} whenever $\phi$ is not of dihedral type, we have in this case the following convexity bound (see for example \cite[(5.20)]{IwaniecKowalskiMR2061214} and of \cite[Proposition 3.3]{SpehMullerMR2053600}) for $L(s, \Ad(\phi))$ on vertical lines $\Re s = \sigma$ where $0 < \sigma < 1$:
\begin{equation}\label{convexitybound}
    L(\sigma + it, \Ad(\phi)) \ll_{\epsilon} \big[Q_{\phi}\cdot (1 + |t|)^3\big]^{\frac{1-\sigma}{2} + \epsilon}.
\end{equation}

Now we prove Theorem 1.
We only need to determine an upper bound for $y \geq 2$, where $y$ is taken such that for any prime number $p \leq y$ with $p\nmid N_1N_2$, at least one of $|\lambda_{\phi_1}(p)|$ and $|\lambda_{\phi_2}(p)|$ is greater than $2$. We may also assume that both $\phi_1$ and $\phi_2$ are not dihedral, as otherwise either one of $\phi_1$ or $\phi_2$ already satisfies the Ramanujan conjecture at all unramified primes.

For this purpose, as in \cite{RamakrishnanMR1453061} and \cite{LuoZhouMR3928043}, we firstly pick up the crucial observation that: if the Ramanujan conjecture is not satisfied at a prime $p\nmid N_i$, then the (twisted) Hecke eigenvalue shall be large at $p^{2}$ by satisfying: $\lambda_{\phi_i}(p^2)\overline{\chi_{\phi_i}(p)} > 3$, $i=1,2$. This follows directly from the Hecke relation
\begin{equation}\label{>3atp^2}
    \lambda_{\phi_i}(p^2)\overline{\chi_{\phi_i}(p)} = |\lambda_{\phi_i}(p)|^2 - 1
\end{equation}
in which the right hand side will be greater than 3 if $|\lambda_{\phi_i}(p)| > 2$, and is not less than $-1$ at all time as $|\lambda_{\phi_i}(p)|^2 \geq 0$. We shall hence consider a relevant summation involving $\lambda_{\phi_i}(p^2)\overline{\chi_{\phi_i}(p)}$'s  whose length is controlled by a parameter $x$ which is associated to $y$. Denoting this summation as $S(x)$ for now, we determined an upper bound for $S(x)$ using the analytic properties of the adjoint sqaure $L$-functions for $\phi_1$ and $\phi_2$, and a lower bound by utilizing the observation \eqref{>3atp^2} and the sieving machinery in \cite{MatomakiMR2887873}.

Let $y$ be as before such that for any prime number $p \leq y$ we have $|\lambda_{\phi_1}(p)| > 2$ or $|\lambda_{\phi_2}(p)| > 2$, and we set $x = y^U$ where $U \geq 1$ is to be determined later. We define $B(\cdot)$ as the Dirichlet convolution of $A_{\phi_1}(\cdot)$ and $A_{\phi_2}(\cdot)$:  
\[
    B(n) := [A_{\phi_1} * A_{\phi_2}](n) 
    = \sum_{r|n}A_{\phi_1}(r)A_{\phi_2}(n/r).
\]
Hence $B(\cdot)$ is again multiplicative, and 
\[
    \sum_{\substack{n\geq 1}}\frac{B(n)}{n^s} = L(s, \Ad(\phi_1))\cdot L(s, \Ad(\phi_2))
\]
for $\Re s > 1$.
In particular, for $p\nmid N_1N_2$, 
\begin{equation*}\label{B()atprime}
    B(p) = A_{\phi_1}(p) + A_{\phi_2}(p) = \lambda_{\phi_1}(p^2)\overline{\chi_{\phi_1}(p)} + \lambda_{\phi_2}(p^2)\overline{\chi_{\phi_2}(p)}
\end{equation*}
as $A_{\phi_{i}}(p) = \lambda_{\phi_i}(p^2)\overline{\chi_{\phi_i}(p)}$ for $i = 1,2$.
By \eqref{>3atp^2} we have $B(p) > 3 + (-1) = 2$ if \eqref{simRCequation} fails, and $B(p) \geq -2$ in all cases.

We shall be considering the following:
\[
    S(x) := \psum_{\substack{1 \leq n \leq x \\ (n,N_1N_2)=1}}B(n)\log(x/n),
\]
where $\psum$ denotes summing over square-free numbers.

We shall derive an upper bound for $S(x)$ of the shape
\[
    x^{\sigma+\epsilon}\cdot Q_{\phi_1}^{\frac{1-\sigma}{2}+\epsilon}Q_{\phi_2}^{\frac{1-\sigma}{2}
    +\epsilon}
\]
where $\sigma < 1$ for $\epsilon >0$ small,
as well as a lower bound 
\[
    S(x) \gg_U x(\log x)^2
\]
where $U \geq 1$ is a suitable constant to be determined.
From the these bounds Theorem \ref{1Theorem1smallestRCprime} follows directly.

\subsection{Upper bound for $S(x)$}
Using the analytic properties of $L(s, \Ad(\phi_i))$'s, we shall prove the following:
\begin{lemma}\label{UpperBoundforS(x)}
    We have for any $\frac{23}{32} + \epsilon< \sigma < 1$: 
    \[
        S(x) \ll_{\epsilon} x^{\sigma+\epsilon}\cdot Q_{\phi_1}^{\frac{1-\sigma}{2}+\epsilon}Q_{\phi_2}^{\frac{1-\sigma}{2}+\epsilon}
    \]
    where the implied constant depends only on $\epsilon > 0$.
\end{lemma}

\begin{proof}
    We define the following auxiliary Euler product:
    \begin{align*}
        G(s)  & := \prod_{p\nmid N_1N_2} \bigg[\bigg(1 - \frac{A_{\phi_1}(p)}{p^s} + \frac{A_{\phi_1}(p)}{p^{2s}} - \frac{1}{p^{3s}}\bigg) \bigg(1 - \frac{A_{\phi_2}(p)}{p^s} + \frac{A_{\phi_2}(p)}{p^{2s}} - \frac{1}{p^{3s}}\bigg)\bigg]\cdot \bigg[1 + \frac{B(p)}{p^s}\bigg] \\
        & \times \prod_{\substack{p\nmid N_1\\p|N_2}}\bigg(1 - \frac{A_{\phi_1}(p)}{p^s} + \frac{A_{\phi_1}(p)}{p^{2s}} - \frac{1}{p^{3s}}\bigg)\cdot 
        \prod_{\substack{p\nmid N_2\\p|N_1}}\bigg(1 - \frac{A_{\phi_2}(p)}{p^s} + \frac{A_{\phi_2}(p)}{p^{2s}} - \frac{1}{p^{3s}}\bigg)
    \end{align*}
    By \eqref{KimSarnakboundforphi} and \eqref{>3atp^2}, we have 
    \begin{equation}\label{KimSarnakBoudnforAp}
        |A_{\phi_i}(p)| \leq p^{\frac{7}{32}} + p^{\frac{-7}{32}} + 1.
    \end{equation}
    Then, upon expanding the factors and making use of \eqref{KimSarnakBoudnforAp}, we see that for $p\nmid N_1N_2$ we have
    \begin{align*}
        \bigg[\bigg(1 - \frac{A_{\phi_1}(p)}{p^s} + \frac{A_{\phi_1}(p)}{p^{2s}} - \frac{1}{p^{3s}}\bigg) \bigg(1 - \frac{A_{\phi_2}(p)}{p^s} + \frac{A_{\phi_2}(p)}{p^{2s}} - \frac{1}{p^{3s}}\bigg)\bigg]\cdot \bigg[1 + \frac{B(p)}{p^s}\bigg] 
        = 1 + O(p^{-2\sigma + \frac{7}{16}}),
    \end{align*}
    where $\sigma := \Re(s)$.
    We therefore conclude that for $\sigma > \frac{23}{32}+\epsilon$, $G(s)$ converges absolutely and we have
    \[
        |G(s)| \ll \tau(N_1)\tau(N_2)\cdot \prod_{p\nmid N_1N_2}(1 + O(p^{-1 - 2\epsilon})) \ll_{\epsilon} (N_1N_2)^{\epsilon},
    \]
    where the implied constant depends only on $\epsilon > 0$ and $\tau(\cdot)$ is the divisor function.
    
    Now, as for $\sigma > 1$ we have 
    \[
        L(s,\Ad(\phi_1))L(s,\Ad(\phi_2))G(s) = \psum_{\substack{n \geq 1 \\ (n,N_1N_2)=1}}\frac{B(n)}{n^s}
    \]
    by \eqref{EulerProductofAd}, for $c>1$ we could relate $S(x)$ to a line integral as
    \begin{align*}
        S(x) & = \frac{1}{2\pi i}\int_{(c)}L(s,\Ad(\phi_1))L(s,\Ad(\phi_2))G(s)\frac{x^s}{s^2}ds.
    \end{align*}
    As $L(s, \Ad(\phi_1)), L(s, \Ad(\phi_2))$ are continued and $G(s)$ is absolutely convergent for $\sigma > \frac{23}{32} + \epsilon$, we may shift the integral line left to any vertical line with real part $\sigma$ satisfying $\frac{23}{32} +\epsilon < \sigma <1$ and conclude
    \[
        S(x) \ll_{\epsilon} x^{\sigma+\epsilon}\cdot (Q_{\phi_1}Q_{\phi_2})^{\frac{1-\sigma}{2} + \epsilon}
    \]
    in the view of the convexity bounds for $\Ad(\phi_1)$ and $\Ad(\phi_2)$ as in \eqref{convexitybound}.
    
\end{proof}

\subsection{Lower bound for $S(x)$}
Following \cite[Lemma A.2, Lemma A.3]{LuoZhouMR3928043}, we define a multiplicative function $h(\cdot)$ supported on square-free positive integers, by setting $h(1) = 1$ and
\[
h(p) =
\begin{cases}
  2,  & p \le y, \\
 -2,  & p > y
\end{cases}
\]
 and then multiplicatively extend to all square-free numbers. We set $h(n) = 0$ if $n$ is not square-free.
    
In the following lemma, we transform the problem of bounding $S(x)$ from below to that of bounding a summation of $h(n)$. 
\begin{lemma}\label{LowerBoundLemma}
    Let $z > 1$.
    If for all $t \leq z$ and positive integers $r \leq z$, we have $\sum_{\substack{n\leq t \\ (n,rN_1N_2)=1}}h(n) \geq 0$, then $S(z) \geq \sum_{\substack{n \leq z \\ (n,N_1N_2)=1}}h(n)\log(z/n)$.
\end{lemma}

\begin{proof}
    We firstly work with $S_0(z)$, which is defined as
    \[
        S_0(z) := \psum_{\substack{1 \leq n \leq z \\ (n,N_1N_2)=1}}B(n).
    \]
    In view of the Mobius inversion, we define a new multiplicative function $g(\cdot)$ by
    \[
        B(n) = \sum_{d|n}h(d)g(n/d).
    \]
    In particular $g(1) = 1$ and $g(p) = B(p) - h(p) \geq 0$ for $p\nmid N_1N_2$. 

    Thus for any $1 < t \leq z$
    \begin{align*}
        S_0(t) & = \psum_{\substack{n\leq t \\ (n,N_1N_2)=1}} \sum_{d|n}h(d)g(n/d) \\
        & = \psum_{\substack{r\leq t \\ (r,N_1N_2)=1}}g(r)\sum_{\substack{d\leq t/r\\ (d,rN_1N_2)=1}}h(d) \\
        & \geq \sum_{\substack{n\leq t \\ (n,N_1N_2)=1}}h(n),
    \end{align*}
    where in the last step by the positivity of $g(\cdot)$ over square-free numbers and the positivity assumption on $h(\cdot)$ we dropped all the terms in the double summation except for that corresponding to $r = 1$.

    Finally, as $S(z) = \int^{z}_1S_0(t)\frac{dt}{t}$ and 
    $\sum_{\substack{n\leq z\\ (n,N_1N_2)=1}}h(n)\log(z/n) = \int^{z}_1\big(\sum_{\substack{n\leq t\\ (n,N_1N_2)=1}}h(n)\big)\frac{dt}{t}$
    by partial summation, we conclude 
    \[
        S(z) \geq \sum_{\substack{n\leq z \\ (n,N_1N_2)=1}}h(n)\log(z/n).
    \]
\end{proof}

The following lemma guarantees the required positivity in the assumption of Lemma \ref{LowerBoundLemma} by providing an asymptotic formula for mean value of $h(\cdot)$.

\begin{lemma}\label{AsymptoticFormulasforh()}
    Let $U \geq 1$ and $y,h(\cdot)$ be as before, and let $r \leq y^U$ be a positive integer, we then have 
    \[
        \sum_{\substack{n\leq y^u \\ (n, 
        rN_1N_2)=1}}h(n) = c(rN_1N_2)\cdot \big[\sigma_2(u) + o_{U}(1)\big](\log y)y^u
    \]
    uniformly for $u \in [U^{-1}, U]$, where $\lim_{y\rightarrow \infty}o_U(1) = 0$ and 
    \[
        c(a) = \frac{\varphi(a)^2}{a^2}\prod_{p\nmid a}(1-1/p)^2(1+2/p) \gg (\log\log a)^{-2}.
    \]
    The continuous function $\sigma_2(u)$ is the solution to the differential-difference equation
    \begin{align*}
       & \sigma_2(u) = u, \,\, 0<u\leq 1, \\
       &  \big(u^{-1}\sigma_2(u)\big)' = -\frac{4}{u^2}\sigma_2(u-1),\,\, u> 1.
    \end{align*}
\end{lemma}

\begin{proof}
    This is \cite[Lemma 6]{MatomakiMR2887873}, in whose notation we take: $K = 1, x_0 = 0, x_1 = 1, \chi_0 = 2, \chi_{1} = -2$. 

    We also remark here that by \cite[Remark 7]{MatomakiMR2887873} and \cite[Lemma 4.2]{LauLiuWuMR2901268}, the original co-primary condition $(n,r)=1$ could be modified to $(n,rN_1N_2)=1$ for $r \leq y^U$, as doing so only creates an extra factor which could be bounded by $(N_1N_2)^{\epsilon}$ for arbitrarily small $\epsilon > 0$ in the right hand side of the weaker assumption prescribed in \cite[Remark 7]{MatomakiMR2887873}, which is admissible for our purpose.
\end{proof}

Lastly, recalling $x = y^U$, we link $\sum_{n\leq x}h(n)$ and $\sum_{n\leq x}h(n)\log (x/n)$ together via partial summation. The lemma below is essentially \cite[Lemma A.5]{LuoZhouMR3928043}.

\begin{lemma}\label{h()partialsummaion}
Let $U \geq 1$ be such that $\sigma_2(u) > 0$ for $1 < u \leq U$, then for
$y\gg_U 1$ we have 
\[
    \sum_{\substack{n \leq y^U \\ (n,N_1N_2)=1}}h(n)\log(y^U/n) \gg_U y^U(\log y)^2.
\]
\end{lemma}

\begin{proof}
    Defining $H(t) := \sum_{\substack{n \leq t \\ (n,N_1N_2)=1}}h(n)$, by partial summation we have that 
    \begin{align*}
        \sum_{\substack{n \leq y^U \\ (n,N_1N_2)=1}}h(n)\log(y^U/n) & = \int^{y^U}_{1}H(t)\frac{dt}{t} \\
        & = \int^{U}_{0}H(y^u)\log(y) du \geq \int^{U}_{U^{-1}}H(y^u)\log (y)du,
    \end{align*}
    where the last inequality follows from the fact that for $p \leq y$ and $p\nmid N_1N_2$, $h(p) =2$
    and $U\geq 1$.
    On the other hand, uniformly for $U^{-1} \leq u \leq U$ Lemma \ref{AsymptoticFormulasforh()} gives
    \[
        H(y^u) = c(N_1N_2)(\sigma_2(u) + o_U(1))\cdot (\log y)y^u.
    \]
    Thus for $y \gg_U 1$ large, we conclude
    \[
        \int^{U}_{U^{-1}}H(y^u)\log (y)du \gg_U c(N_1N_2)y^U(\log (y))^{2}
    \]
\end{proof}

Notice that the above implies that if $U$ is taken such that the assumptions in Lemma \ref{h()partialsummaion} are satisfied (notice that by Lemma \ref{AsymptoticFormulasforh()} the assumptions in Lemma \ref{LowerBoundLemma} are also satisfied with $z = x = y^U$), then for $y$ larger than a constant (absolute once $U$ is fixed), we may invoke Lemma \ref{LowerBoundLemma} to conclude 
\[
    S(x) = S(y^U) \geq \sum_{\substack{n \leq y^U \\ (n,N_1N_2)=1}}h(n)\log(y^U/n) \gg y^U(\log y)^2 \gg x(\log x)^2.
\]
This provides the desired lower bound for $S(x)$.

\subsection{Proof of Theorem \ref{1Theorem1smallestRCprime}}
We are now ready to prove Theorem 1.
\begin{proof}
    We recall that for $u \in (0,1]$ we have $\sigma_2(u) > 0$, and a numerical computation using \textit{Mathematica} shows that the first positive 
zero of $\sigma_2(u)$ is 2.23528..., and we thus take $U$ to be $U = 2.23527$. By the intermediate value theorem for continuous functions, $\sigma_2(u)$ is positive on the interval $(0, 2.23527]$ and hence so do $H(y^U)$ given $y$ larger than a computable absolute constant.
Now, combining Lemma \ref{UpperBoundforS(x)}, Lemma \ref{LowerBoundLemma} and Lemma \ref{h()partialsummaion},
we conclude that 
\[
    y = x^{\frac{1}{U}} \ll_{\epsilon} (Q_{\phi_1}Q_{\phi_2})^{\frac{1}{2U}+\epsilon} \ll [N_1N_2(1 + |t_{\phi_1}|)(1 + |t_{\phi_2}|)]^{0.447374},
\]
where the implied constant in the last "$\ll$" is absolute once we fix an $\epsilon > 0$ small enough, completing the proof of Theorem \ref{1Theorem1smallestRCprime}.    
\end{proof}

\subsection{Proof of Theorem \ref{2Theorem1smallestRCprime}}
We then turn to the proof of Theorem 2, which is parallel to that of Theorem 1 with slight modifications.
\begin{proof}
Letting $\phi_i, i = 1,2,3$ be three primitive Hecke-Maass cusp forms as before,
    then we have 
    \[
        \sum_{i=1}^3 \lambda_{\phi_i}(p^2)\overline{\chi_{\phi_i}(p)} > 1.
    \]
    Setting $y$ similarly as before, defining $h(\cdot)$ as 
    \[
        h(p) = \begin{cases}
                1, &  p \leq y \\
                -3, & p > y
                \end{cases}
    \]
    for $p \nmid N_1N_2N_3$, and relating to $\prod_{i=1}^3L(s, \Ad(\phi_i))$ for the upper bound 
    for the relevant sum and shift the integral line in Lemma \ref{UpperBoundforS(x)} to $\Re(s) = \sigma$ with $\frac{53}{64} +\epsilon < \sigma < 1$, we shall be able to find the prescribed prime in Theorem \ref{2Theorem1smallestRCprime} with 
    \[
        p \ll \big[\prod_{i=1}^3N_i(1+|t_{\phi_i}|)\big]^{0.778798}
    \]
    where the implied constant is absolute.
    This completes the proof of Theorem \ref{2Theorem1smallestRCprime}.
\end{proof}
    
\begin{remark}
    However, for $l = |\{\text{primitive Hecke-Maass cusp form}\}| \geq 4$, the above argument 
    fails, as we are no longer able to guarantee that 
    \[ 
        \sum_{i=1}^{l}\lambda_{\phi_{i}}(p^2)\overline{\chi_{\phi_i}(p)} > 0
    \]
    when at least one of $\{|\lambda_{\phi_i}(p)|\}_{i=1}^l$'s greater than 2.
\end{remark}

\section{Proof of Theorem \ref{main theorem density two forms} and \ref{main theorem density finitely many forms}}
Throughout this section, we assume that all the involved primitive Hecke-Maass cusp forms  are not of dihedral, tetrahedral or octahedral type, as otherwise the Ramanujan conjecture is already satisfied at every unramified primes.

\subsection{Preliminaries }

 We introduce the following notation for the twisted symmetric fourth power of a primitive Hecke-Maass cusp form $\phi$ of level $N$ with central character $\chi_{\phi}$:
\[
    A^4(\phi) := \Sym^4(\phi)\otimes \overline{\chi_{\phi}}^2.
\]
 We also denote the usual symmetric $m$-th power of $\phi$ by $\Sym^m(\phi)$. In particular, under our assumption that $\phi$ is not of dihedral, tetrahedral or octahedral type, $\Sym^3(\phi)$ and $A^4(\phi)$ correspond to cuspidal automorphic representations on $GL_4(\mathbb{A}_{\mathbb{Q}})$ and $GL_5(\mathbb{A}_{\mathbb{Q}})$ respectively (see \cite{Kim2003},\cite{KimShahidi2002} and \cite{KimShahidi2MR1923967}). Denoting the local parameters for $\phi$ at $p\nmid N$ to be $\{\alpha_p, \beta_p\}$, then the local parameters for $\Ad(\phi)$, $\Sym^3(\phi)$ and $A^4(\phi)$ at $p$ are respectively given by 
 \[
    \{\alpha_p/\beta_p, 1, \beta_p/\alpha_p\}\,,\{\alpha_p^3, \alpha_p^2\beta_p, \alpha_p\beta_p^2, \beta_p^3\},
 \]
 and 
 \[
    \{\alpha_p^2/\beta_p^2, \alpha_p/\beta_p, 1, \beta_p/\alpha_p, \beta_p^2/\alpha_p^2\}.
 \]
 Just as $\Ad(\phi)$, we may define the $L$-functions for $\Sym^3(\phi)$ and $A^4(\phi)$ by Euler products, which could be further expanded as Dirichlet series:
 \begin{align*}
     L(s,\Sym^3(\phi)) & := \prod_{p\nmid N}\big[(1-\alpha_p^3p^{-s})(1-\alpha_p^2\beta_pp^{-s})(1-\alpha_p\beta_p^2p^{-s})(1-\beta_p^3p^{-s})
     \big]^{-1} \\
     & = \sum_{\substack{n\geq 1 \\ (n,N)=1}}A^{[3]}_{\phi}(n)n^{-s}
 \end{align*}
 and 
  \begin{align*}
     L(s,A^4(\phi)) & := \prod_{p\nmid N}\big[(1-\alpha_p^2\beta_p^{-2}p^{-s})(1-\alpha_p\beta_p^{-1}\beta_pp^{-s})(1-p^{-s})(1-\beta_p\alpha_p^{-1}p^{-s})(1-\beta_p^2\alpha_p^{-2}p^{-s})
     \big]^{-1} \\
     & = \sum_{\substack{n\geq 1 \\ (n,N)=1}}A^{[4]}_{\phi}(n)n^{-s},
 \end{align*}
 for $\Re(s) > 1$.

To prove Theorem \ref{main theorem density two forms} and Theorem \ref{main theorem density finitely many forms}, we firstly establish some necessary lemmas. 
\begin{lemma}\label{small lemma}
    Let $\phi_1,\phi_2$ be primitive Hecke-Maass cusp forms of level $N_1$ and $N_2$. Assuming $\{p:p \nmid N_2,p \in \RP(\phi_1)\} \neq \{p:p \nmid N_1, p \in \RP(\phi_2)\}$, we have $\phi_1 \neq \phi_2$, $\Ad(\phi_1) \neq \Ad(\phi_2)$, and $A^4(\phi_1) \neq A^4(\phi_2)$. 
\end{lemma}

\begin{proof}
    The assertion that $\phi_1 \neq \phi_2$ is trivial. Let $\alpha_p^{(i)},\beta_p^{(i)}$ be the local parameters of $\phi_i$ at $p\nmid N_1N_2$, $i=1,2$. If $\Ad(\phi_1) = \Ad(\phi_2)$ were to hold, then for every prime $p\nmid N_1N_2$ we have 
    $$\left\{\frac{\alpha_p^{(1)}}{\beta_p^{(1)}},1,\frac{\beta_p^{(1)}}{\alpha_p^{(1)}}\right\} = \left\{\frac{\alpha_p^{(2)}}{\beta_p^{(2)}},1,\frac{\beta_p^{(2)}}{\alpha_p^{(2)}}\right\}.$$
    However, by \cite[Lemma 2.1]{LuoZhouMR3928043} we know $p \notin \RP(\phi_i) \Leftrightarrow \frac{\alpha_p^{(1)}}{\beta_p^{(1)}} \in (0,1) \cup (1,\infty)$. It follows that $\RP(\phi_1) = \RP(\phi_2)$, which is a contradiction. Therefore $\Ad(\phi_1) \neq \Ad(\phi_2)$. By the same argument we also conclude $A^4(\phi_1) \neq A^4(\phi_2)$.
\end{proof}

Note that if $\Sym^{2n}(\phi_1)$ and $\Sym^{2n}(\phi_2)$ were known to exist as cuspidal automorphic representations on $\GL_{2n+1}(\mathbb{A}_\mathbb{Q})$, then we also have $\Sym^{2n}(\phi_1) \otimes \overline{\chi_{\phi_1}}^n \neq \Sym^{2n}(\phi_2) \otimes \overline{\chi_{\phi_2}}^n$. 

\begin{lemma}\label{prime number theorem cusp form}
    Let $\phi$ be a primitive Hecke-Maass cusp form of level $N$. Then as $X \to \infty$ we have 
    \begin{align*}
        \sum_{\substack{p \leq X \\ p\nmid N}} A_{\phi}(p) = o(\pi(X)), \\
        \sum_{\substack{p \leq X \\ p\nmid N}} A^{[4]}_{\phi}(p) = o(\pi(X)), \\
        \limsup_{X \to \infty} \frac{\sum_{\substack{p \leq X \\ p\nmid N}} |A^{[3]}_{\phi}(p)|^2}{\pi(X)} \leq 1, \\
        \limsup_{X \to \infty} \frac{\sum_{\substack{p \leq X \\ p\nmid N}} |A^{[4]}_{\phi}(p)|^2}{\pi(X)} \leq 1.
    \end{align*}
\end{lemma}

\begin{proof}
    The first two relations follow from the fact that $\Ad(\phi)$ and $A^4(\phi)$ cuspidal automorphic representations on $\GL_3(\bA_\bQ)$ and $\GL_5(\bA_\bQ)$ respectively, and the Prime Number Theorem for automorphic $L$-functions \cite[Theorem 5.13]{IwaniecKowalskiMR2061214}. The last two relations are \cite[Remark 3.3]{LuoZhouMR3928043}.
\end{proof}

The key new input of our work is the following lemma:

\begin{lemma}\label{prime number theorem rankin-selberg}
    Let $\phi_1$ and $\phi_2$ be primitive Hecke-Maass cusp forms of levels $N_1$ and $N_2$. Assuming $\{p:p \nmid N_2,p \in \RP(\phi_1)\} \neq \{p:p \nmid N_1, p \in \RP(\phi_2)\}$, then as $X \to \infty$ we have 
    \begin{align*}
        \sum_{\substack{p \leq X \\ p\nmid N_1N_2}} A_{\phi_1}(p)A_{\phi_2}(p) = o(\pi(X)), \\
        \sum_{\substack{p \leq X \\ p\nmid N_1N_2}} A_{\phi_1}(p)A^{[4]}_{\phi_2}(p) = o(\pi(X)).
    \end{align*}
\end{lemma}

\begin{proof}
    Let $\chi$ be an Hecke character on $\bA_\bQ^\times$. Since we have assumed $\{p:p \nmid N_2,p \in \RP(\phi_1)\} \neq \{p:p \nmid N_1, p \in \RP(\phi_2)\}$, by Lemma \ref{small lemma} we have $\phi_1 \neq \phi_2$ and $A^4(\phi_1) \neq A^4(\phi_2)$. Under these conditions, by the recent work of Thorner and Zhao \cite[Theorem 1.3 (1a)(3)]{ThornerZhao2026}, the Rankin-Selberg $L$-functions $L(s,\Ad(\phi_1) \times \Ad(\phi_2) \otimes \chi)$ and $L(s,\Ad(\phi_1) \times A^4(\phi_2) \otimes \chi)$ have no Landau-Siegel zeros relative to an absolute constant $c_1>0$. Explicitly, these $L$-functions are non-vanishing in the following real interval:
    $$\left(1-\frac{c_1}{\log (Q_{\phi_1}Q_{\phi_2}Q_\chi)},1\right).$$
    Here $Q_\chi$ is the analytic conductor of $\chi$. In particular, taking $\chi = |\cdot|^{it}, t \in \mathbb{R}$ to be the Archimedean characters, we conclude that the Rankin-Selberg $L$-functions $L(s,\Ad(\phi_1) \times \Ad(\phi_2))$ and $L(s,\Ad(\phi_1) \times A^4(\phi_2))$ are non-vanishing in the following standard zero-free region:
    $$\left\{s = \sigma+it:\sigma \geq 1-\frac{c_2}{\log (Q_{\phi_1}Q_{\phi_2}(|t|+3))} \right\},$$
    where $c_2 > 0$ is another absolute constant. The lemma now follows from the Prime Number Theorem for automorphic $L$-functions \cite[Theorem 5.13]{IwaniecKowalskiMR2061214}.
\end{proof}

\subsection{Proof of Theorem \ref{main theorem density two forms}}
Now we are ready to prove Theorem \ref{main theorem density two forms}:

\begin{proof}
    For a prime $p\nmid N_1N_2$ we define 
    $$U(p) := (1+3A_{\phi_1}(p)+3A_{\phi_2}(p)+5A^{[4]}_{\phi_1}(p))^2,$$
    and consider the partial sum $\sum_{\substack{p \leq X \\ p\nmid N_1N_2}} U(p)$. 
    
    On one hand, by \cite[Lemma 2.1]{LuoZhouMR3928043}, we have 
    $$p \notin \RP(\phi_1) \cup \RP(\phi_2) , p\nmid N_1N_2\Rightarrow A_{\phi_1}(p),A_{\phi_2}(p)>3, \hspace{3mm} A^{[4]}_{\phi_1}(p)>5.$$
    Thus we have 
    \begin{equation}\label{lower bound of U(p)}
        \sum_{\substack{p \leq X \\ p\nmid N_1N_2}} U(p) \geq 44^2 \sum_{\substack{p \leq X \\ p \notin \RP(\phi_1) \cup \RP(\phi_2)\\p\nmid N_1N_2}} 1. 
    \end{equation}

    On the other hand, by the Hecke relations we have for $p\nmid N_1N_2$: 
    \begin{align*}
        A_{\phi_i}(p)^2 &= A^{[4]}_{\phi_i}(p) + A_{\phi_i}(p) + 1, \hspace{3mm} i=1,2, \\
        A_{\phi_1}(p)A^{[4]}_{\phi_1}(p) &= |A_{\phi_1}^{[3]}(p)|^3 - 1.
    \end{align*}
    Therefore we can write
    \begin{align*}
        U(p) &= -11 + 15A_{\phi_1}(p) + 15A_{\phi_2}(p) + 19A^{[4]}_{\phi_1}(p) + 9A^{[4]}_{\phi_2}(p) \\
        &+ 18A_{\phi_1}(p)A_{\phi_2}(p) + 30A^{[4]}_{\phi_1}(p)A_{\phi_2}(p) \\
        &+30|A^{[3]}_{\phi_1}(p)|^2+25|A^{[4]}_{\phi_1}(p)|^2.
    \end{align*}
    Thus by Lemma \ref{prime number theorem cusp form} and Lemma \ref{prime number theorem rankin-selberg} we have 
    \begin{equation}\label{upper bound of U(p)}
        \limsup_{X \to \infty} \frac{\sum_{\substack{p \leq X \\ p\nmid N_1N_2}}U(p)}{\pi(X)} \leq 44. 
    \end{equation}

    Finally, combining \eqref{lower bound of U(p)} and \eqref{upper bound of U(p)} we have 
    $$\overline{d}((\RP(\phi_1) \cup \RP(\phi_2))^c) \leq \frac{1}{44},$$
    which is equivalent to 
    $$\underline{d}(\RP(\phi_1) \cup \RP(\phi_2)) \geq \frac{43}{44}.$$
    This completes the proof of Theorem \ref{main theorem density two forms}.
\end{proof}

\subsection{Proof of Theorem \ref{main theorem density finitely many forms}}
The proof of Theorem \ref{main theorem density finitely many forms} is almost identical to that of Theorem \ref{main theorem density two forms}. One works with 
$$U(p) := (1+3\sum_{j=1}^m A_{\phi_j}(p)+5A^{[4]}_{\phi_1}(p))^2$$
instead. The constant $26+9m$ arises from 
$$26+9m = 1+3\sum_{j=1}^m 3 + 5^2.$$
This completes the proof of Theorem \ref{main theorem density finitely many forms}. $\square$

\begin{remark}
     If we can establish a standard zero-free region for the Rankin-Selberg $L$-function $L(s,A^4(\phi_1) \times A^4(\phi_2))$, then we can elaborate Theorem \ref{main theorem density finitely many forms} to 
\begin{equation}
     \underline{d}(\cup_{j=1}^m \RP(\phi_j)) \geq 1-\frac{1}{1+9m+25m^2}
    \end{equation}
by considering 
$$U(p) = (1+3\sum_{j=1}^mA_{\phi_j}(p)+5\sum_{j=1}^mA_{\phi_j}^{[4]}(p))^2.$$
However, this is currently out of reach. 
\end{remark}

\bigskip 

\bibliographystyle{alpha}
\bibliography{Bib.bib}

\end{document}